\documentclass[oneside,leqno,draft,11pt]{amsart}
\usepackage{latexsym,amssymb,amsmath,graphics,color}

\def\E{{\mathbb{E}}}

\def\P{{\mathbb{P}}}

\def\R{{\mathbb{R}}}

\def\X{{\mathbb{X}}}
\def\Y{{\mathbb{Y}}}

\newcommand{\al}{\alpha }
\newcommand{\te}{\theta}
\newcommand{\Te}{\Theta}

\newcommand{\wt}{\widetilde}

\def\8{\infty}

\def\E{\mathbb{E}}
\def\P{\mathbb{P}}

\def\<{\langle}
\def\>{\rangle}

\renewcommand{\d}{\delta}
\renewcommand{\a}{\alpha}

\newcommand{\eps}{\varepsilon}

\newcommand{\ov}{\overline}

\newcommand{\supp}{\mathrm{supp}}

\newtheorem{thm}[equation]{Theorem}

\newtheorem{lem}[equation]{Lemma}

\theoremstyle{definition}

\newtheorem{rem}{Remark}[section]

\numberwithin{equation}{section}

\begin{document}
\title[Heavy tail estimates for  affine type Lipschitz recursions.​​]{​​A simple proof ​of ​​heavy tail estimates for affine type Lipschitz recursions.​}
\author[D. Buraczewski, E. Damek]
{Dariusz Buraczewski, Ewa Damek}
\address{D. Buraczewski, E. Damek\\ Instytut Matematyczny\\ Uniwersytet Wroclawski\\ 50-384 Wroclaw\\
pl. Grunwaldzki 2/4\\ Poland}
\email{dbura@math.uni.wroc.pl\\ edamek@math.uni.wroc.pl}



\thanks{The research was partially supported by the National Science Centre, Poland (grant numbers  UMO-2014/14/E/ST1/00588
and UMO-2014/15/B/ST1/00060)}

\begin{abstract}
We study the affine recursion $X_n = A_nX_{n-1}+B_n$
where $(A_n,B_n)\in \R^+ \times \R $ is an i.i.d. sequence and recursions  $X_n  = \Phi_n(X_{n-1})$ defined by Lipschitz transformations such that $\Phi (x)\geq Ax+B$.
It is known that under appropriate hypotheses the stationary solution $X$ has regularly varying tail, i.e.
$$
\lim_{t\to\8} t^{\a} \P[X>t] = C.
$$ However positivity of $C$ in general is either unknown or requires some additional involved arguments. In this paper we give a simple proof that $C>0$.
This applies, in particular, to the case when Kesten-Goldie assumptions are satisfied.
\end{abstract}

 \maketitle

\section{Introduction}
\subsection{Random difference equation}
Lipschitz iterations considered in this paper are modeled on the affine recursion, usually called in the literature the random difference equation. This is
the Markov process  $\{X_n\}$  on $\R$ defined by the formula
\begin{equation}
\label{eq:1}
X_n = A_nX_{n-1}+B_n,\ \ n\geq 1,
\end{equation}
where $(A_n,B_n)\in \R^+ \times \R $ is a sequence of i.i.d. (independent identically distributed) random variables and $X_0\in \R$ is an initial distribution.
If $\E \log A < 0$ and $\E \log^+|B|<\8$, the sequence $\{X_n\}$ converges in law to a random
variable $X$, which is the unique solution to the random
difference equation
\begin{equation}
\label{eq:2}
X =_d AX+B, \qquad \mbox{$X$ independent
of }(A,B);
\end{equation}
see \cite{V}.  The celebrated result of Kesten \cite{K} and Goldie \cite{G} is the following
 \begin{thm}
 \label{thm: kesten goldie}
 Assume that $\E \log A <0$, $\E A^\a = 1$ for some $\a>0$, $\E[|B|^\a + A^\a \log^+A]<\8$ and the law of $\log A$ is non-arithmetic. Then
 \begin{equation}\label{eq:3}
\lim_{t\to\8}t^\a \P[X>t] = C_+ \qquad \mbox{and} \qquad \lim_{t\to\8} t^\a\P[X<-t] = C_-.
\end{equation}
Moreover,
   $C_\8 = C_++C_->0$ if and only if
\begin{equation}\label{eq: fixed point}
\P[Ax+B=x]<1 \mbox{ for every $x\in \R$} .
\end{equation}
 \end{thm}
The Kesten-Goldie theorem found enormous number of applications, both in pure and applied mathematics. We refer to \cite{BDM} for an overview of recent results related to the process $\{X_n\}$, examples and applications.


 Recently Guivarc'h and Le Page \cite{GL1} improved the second part of Theorem \ref{thm: kesten goldie}, showing that $C_+>0$ (or $C_->0$) if and only if the support of $X$ is unbounded at $+\8$ (or respectively at $-\8$). Thus, it cannot happen that $X$ has nontrivial tails of different order at $+\8$ and $-\8$.
The aim of this note is in particular to provide a simple proof of their result.

Existence of the limit in \eqref{eq:3} follows from a renewal type argument, which now is well understood in a much more general context. However, positivity of the limiting constant $C_\8$ does not follow directly from \eqref{eq:3} and usually requires an extra argument: for the affine recursion see \cite{Gri}, \cite{G}, for Lipschitz recursions see \cite{Mi}. Mirek's proof  (adopted from \cite{BDGHU}) is by no means satisfactory. In many cases, as described below, natural conditions for positivity of $C_\8$ are not known and the problem to formulate such remains open. We are going to present a simple argument that gives not only positivity of $C_\8$, but, what is more important, positivity of $C_+$ (or $C_-$) for a class of Lipschitz recursions including the ``Letac Model''.  This improves ``positivity '' results not only of Goldie \cite{G} and Mirek \cite{Mi}, but also of Guivarc'h and Le Page \cite{GL1} and Collamore and Vidyashankar \cite{CV}.

\subsection{Affine type Lipschitz recursions}
In his paper Goldie studied not only the affine recursion but also some slight modifications of it like the extremal recursion $$X_n  = \max\{A_nX_{n-1},B_n\},\qquad n\ge 1$$
or the Letac model
$$X_n  = \max\{A_nX_{n-1}+B_n, A_nC_n+B_n\},\qquad n\ge 1$$
and he observed that \eqref{eq:3} holds also in this extended setting.
More generally, one can consider the iterated functions system (IFS), i.e. recursions of the type
\begin{equation}\label{eq:4}
  \X_n = \Psi_n(\X_{n-1}),\qquad n\ge 1,
\end{equation}
where $\Psi_n$ is a sequence of i.i.d. random Lipschitz mappings on $\R$. Beginning
from the early nineties IFS modeled on Lipschitz functions have attracted a lot of attention:
Alsmeyer \cite{Alsm2014}, Arnold and Crauel \cite{AC},  Brofferio and Buraczewski \cite{BB},  Diaconis and Friedman \cite{DF}, Duflo \cite{Du}, Elton \cite{Elton}, Hennion and Herv\'e \cite{HH}, Mirek \cite{Mi}.

Sufficient conditions for existence of the stationary distribution were provided by Diaconis, Friedman \cite{DF} and Elton \cite{Elton}.
As in the affine case, $\X_n$ converges in distribution to $\X$, which is the unique solution to the stochastic equation
\begin{equation}
\label{eq: rec psi}
\X =_d \Psi(\X), \qquad \mbox{$\X$ independent
of }\Psi.
\end{equation}
However, to describe the tail of $\X$ some further assumptions are needed.
Usually one assumes that $\Psi(x)$ is close to the affine mapping $Ax+B$, then under the Cramer condition on $A$, Alsmeyer \cite{Alsm2014} and Mirek \cite{Mi} described the tail of $\X$ and proved \eqref{eq:3}. However positivity of the limiting constant was proved only in a very particular cases. Our aim is to  fill this gap.

 \section{Main result}
 \label{sec:main result}
A temporally homogeneous Markov chain $\{\X_n\}_{n\geq 0}$ on $\R$ is called iterated function system of iid Lipschitz mapps  {\it (IFS)}, if it satisfies a recursion of the form
\begin{equation}
\X_n=\Psi (\omega _n, \X_{n-1}),\quad \mbox{for} \ n\geq 1,
\end{equation}
where
\begin{itemize}
\item $\X_0$, $\{\omega_n\}_{n\ge 1}$ are independent random elements on a common probability space $\Omega$,
\item $\{\omega_n\}_{n\ge 1}$ are identically distributed and taking values in a measurable space $\Theta$,
\item $\Psi : \Theta \times \R \mapsto \R$ is jointly measurable and Lipschitz continuous in the second argument i.e.
$$
\big|\Psi (\omega ,x) - \Psi (\omega ,y)\big|\leq C_{\omega }|x-y|,$$
for all $x,y\in \R $, $\omega \in \Theta $ and a suitable $C_{\omega }>0$.
\end{itemize}
We will also write $\X _n= \Psi_n(\X_{n-1})$ for short. Then
$$
\X_n = \Psi_n\circ\ldots\circ \Psi_1(\X_0)=:\Psi _{n,1}(\X_0).
$$
Let $L(\Psi )$, $L(\Psi _{n,1})$ be the Lipschitz constants of $\Psi $, $\Psi _{n,1}$ respectively.
If $\E \log ^+L(\Psi )<\8 $, $\E \log ^+ | \Psi (\omega ,0)|<\8 $  and
$$
\lim _{n\to \8 }\frac{1}{n}\log L(\Psi _{n,1})<0 \quad \mbox{a.s.}$$
then
 $\X_n$ converges in distribution to a random variable $\X$, which does not depend on $\X_0$ and it satisfies \eqref{eq: rec psi}.

In this paper we consider IFS that can be estimated from below by the random affine transformation
 \begin{equation}\label{eq: A0}
   \Psi(x) \ge Ax+B
 \end{equation} for some random pair $(A,B)\in \R^+\times \R$. Without loosing generality we can assume $B< 0$ a.s.
Thus parallel to process $\{\X_n\}$ one can define the iteration  $\ov\X_{n+1}=A_{n+1} \ov\X_n+B_{n+1}$. A simple induction argument proves that for every n
 \begin{equation}\label{eq:6}
   \X_n \ge \ov \X_n\quad a.s.
 \end{equation}
We assume that the random pair $(A,B)$ satisfies hypotheses of Theorem \ref{thm: kesten goldie}. This implies, in particular, that $\ov \X_n$ converges in distribution to $\ov \X$, a solution to \eqref{eq:2}, which must be negative a.s. and
such that
\begin{equation}\label{eq:7}
  \lim _{t\to \8}t^{\a}\P[\ov \X<-t]=  C_1 >0.
\end{equation}
The Cramer condition on $A$ implies, in particular, that
\begin{equation}\label{eq:8}
  \lim_{t\to\8}t^\a\P[M>t] = C_2 > 0
\end{equation}
for $M=\max _n\Pi_n$, $\Pi_n= A_1\ldots A_n$.

Our main result is the following

 \begin{lem}
 \label{mthm}
Suppose that \eqref{eq: A0} and the assumptions of Theorem \ref{thm: kesten goldie} are satisfied. Assume further that $\E \log L(\Psi )<0$ and $\E \log ^+ | \Psi (\omega ,0)|<\8 $. 
 If $\X$ is unbounded at $+\8$, then there is $\epsilon>0$ such that
$$\P[\X>t]> \epsilon t^{-\a}$$ for large $t$.
 \end{lem}
The lemma may seem technical, but as we will see in Section \ref{sec: app}, for concrete examples it reduces the problem of positivity of the limiting constant to unboundedness of the support of the stationary measure.

\begin{rem}
Let $\nu $ be the law of $\X $. Suppose that
\begin{equation}\label{preserve}
\Psi (\supp \nu )\subset (\supp \nu ) \ a.s.,
\end{equation}
which happens for instance when $\Theta $ is a metric space and $\Psi $ is jointly continuous (see section \ref{section: Lipschitz}). Then condition \eqref{eq: A0} may be replaced by a weaker one
 \begin{equation}\label{eq: A1}
   \Psi(x) \ge Ax+B\quad x\in \supp \nu
 \end{equation}
 and the Lemma still holds.
\end{rem}

\begin{proof}[Proof of Lemma \ref{mthm}]
{\bf Step 1.} First we prove a stronger version of inequality \eqref{eq:6}. Let $\underline \omega =(\omega _1, \omega _2, ...)$
be a generic element of $\Omega$ and let $\theta \underline \omega =(\omega _2, \omega _3, ...)$ be the shift operator.

Given a pair $(a,b)\in \R^+\times \R$ we denote by
$$
(a,b)\circ x = ax+b
$$ the affine action of $(a,b)$ on $\R$. Then the process $\{\ov \X_n\}$ can be written as
\begin{align*}
  \ov \X_n^x &  = (A(\omega_n),B(\omega_n))\circ \ldots \circ (A(\omega_1),B(\omega_1))\circ x. \\
   &  = (A_n,B_n)\circ \ldots \circ (A_1,B_1)\circ x.
\end{align*}
 Let $\{\ov \Y^x_n\}$ be the associate backward process defined by
$$
\ov \Y_n^x = (A_1,B_1)\circ \ldots \circ (A_n,B_n)\circ x = \sum_{k=1}^n \Pi_{k-1} B_k + \Pi_n x = \ov\Y_n + \Pi_n x.$$
Let
$$
 \Psi _{1,n} (x)= \Psi_1\circ\ldots\circ \Psi_n(x)
$$
be the backward iteration. If $\E \log L(\Psi )<0$ and $\E \log ^+ | \Psi (\omega ,0)|<\8 $ then
$$
\Y =: \lim _{n\to \8 }\Psi _{1,n} (x)\ \mbox{exists a.s.},$$ see \cite{DF}, and it has the distribution law $\nu $. By \eqref{eq: A0} for every $n$
$$
 \Psi _{1,n} (x)\geq \ov \Y ^x_n.$$ Therefore,
$$
\Y \geq \lim _{n\to \8}\ov Y^x_n=\sum _{k=1}^{\8}\Pi _{k-1}B_k=:\ov \Y  a.s.$$
Then, by \eqref{eq: rec psi} and \eqref{eq: A0}, for every $n$
\begin{equation}\label{eq:10}
  \X = _d \Y =\Psi_1\circ\ldots\circ \Psi_n(\Y (\theta^n(\underline \omega)))
  \ge \ov \Y_n(\omega) + \Pi_n \Y (\theta^n\underline \omega).
\end{equation}
Notice that $\Y(\theta^n\underline \omega)\in \supp \nu $ so if \eqref{preserve} holds then  \eqref{eq: A0} may be replaced by \eqref{eq: A1} and the same argument goes through.

{\bf Step 2.}
Now let
$$
U_n = \big\{ \Pi_n > t \mbox{ and } -Ct < \ov \Y_n \big\}. $$
Using $\ov \Y _n\geq \ov \Y$, we prove that there are $C$ and $\d>0$ such that  for large $t$
\begin{equation}\label{eq:un}
\P\bigg[\bigcup_n U_n \bigg] \ge \d t^{-\a}.
\end{equation}
 By \eqref{eq:7} and \eqref{eq:8}, for large $t$, we have
\begin{eqnarray*}
\frac{C_2}2 t^{-\a} &\le& \P[M > u] \\
&=& \P\big[ \Pi_n > t \mbox{ for some $n$}  \big]\\
&=& \P\big[ \Pi_n > t \mbox{ for some $n$ and }  \ov \Y \le - Ct  \big] + \P\big[ \Pi_n > t \mbox{ for some $n$ and } \ov \Y   >  - Ct  \big]\\
&\le& \frac{2C_1}{C^\a}t^{-\a} + \P\big[ \Pi_n > t \mbox{ and } \ov \Y_n > -Ct \mbox{ for some $n$}\big]\\
&\le& \frac{2C_1}{C^\a}t^{-\a}+ \P\bigg[ \bigcup_n U_n \bigg].
\end{eqnarray*}
Choosing large $C$ enough we obtain \eqref{eq:un}.

{\bf Step 3.}
Since $\Y$ is unbounded at $+\8$
\begin{equation}
\P [\Y >C+1]=\eta >0.
\end{equation}
In view of \eqref{eq:7},  \eqref{eq:10} and \eqref{eq:un}, for large $t$ we have
\begin{eqnarray*}
\frac{C_2}2\d \eta t^{-\a} &<& \eta \d \P[M>t]\\ &\le & \eta \P\bigg[\bigcup_n U_n\bigg]\\
 &=  & \eta \sum_n \P\bigg[U_n \cap \bigg(\bigcup_{k=1}^{n-1} U_k\bigg)^c \bigg]\\
 &\le   & \sum_n \P\bigg[ \big\{  \Pi_n(\underline{\omega}) > t, \ov \Y_n(\underline{\omega}) > -Ct  \big\}\cap  \bigg(\bigcup_{k=1}^{n-1} U_k\bigg)^c \bigg]
 \cdot \P[ \Y (\theta^n\underline \omega ) > C+1] \\
 & \le   & \sum_n \P\bigg[ \big\{  \Pi_n (\underline \omega )\Y (\theta^n\underline \omega )+ \ov \Y_n (\underline \omega) > t  \big\}\cap U_n \cap \bigg(\bigcup_{k=1}^{n-1} U_k\bigg)^c \bigg]\\
& \le    & \sum_n \P\bigg[ \big\{  \Y > t  \big\}\cap U_n  \cap \bigg(\bigcup_{k=1}^{n-1} U_k\bigg)^c \bigg]\\
&\le& \P[\X >t].
\end{eqnarray*} This proves the lemma.
\end{proof}

\section{Applications of Lemma \ref{mthm}}
\label{sec: app}
In this section we study a few examples when $\Theta $ is a metric space and $\Psi $ is jointly continuous. Then the support of $\nu $ - the distribution law of the stationary solution is preserved by $\Psi $ a.s. and so \eqref{eq: A1} suffices to apply Lemma \ref{mthm}.

\subsection{Random difference equation}
\label{sec: random difference}
The following result was proved recently by Guivarc'h and Le Page \cite{GL1}.
 \begin{thm}
   \label{thm: gl}
Let $X$ be as in \eqref{eq:1} and let $\nu $ be the law of $X$. Under assumptions of Theorem \ref{thm: kesten goldie}, $C_+> 0$ if and only if the support of $\nu$ is unbounded at $+\8$.
 \end{thm}
 Lemma \ref{mthm} provides a simple proof of the Guivarc'h - Le Page theorem  and, in particular, gives positivity of the constant $C_\8$.
Indeed, it implies that, if the support of $X$ in unbounded at $+\8$, then
 $$
 \P[X>t] \ge \epsilon t^{-\a}
 $$ for some $\epsilon >0$ and large $t$. Thus the problem of positivity of $C_+$ in \eqref{eq:3} is reduced to boundedness or unboundedness of the support of $\nu$ at $+\8$.


If \eqref{eq: fixed point} holds then $\supp \nu $ is unbounded. Indeed, there are at least two points $x,y\in {\rm supp\; } \nu$. Moreover, $\E A^{\a }=1$ implies that there is $(a,b)\in {\rm supp\;}\mu$ with $a>1$. Since the support of $\nu$ is $\supp\mu$-invariant, points $(a,b)^n\circ x$ and $(a,b)^n\circ y$ are elements of the support of $\nu$. But their distance
 $$
\big| (a,b)^n\circ x - (a,b)^n\circ y\big| = a^n|x-y|
 $$ converges to $+\8$. Therefore the support of $\nu$ must be unbounded.

In fact there is a more precise description of $\supp \nu $.
If \eqref{eq: fixed point} holds then the support of $\nu$ is either $\R$, or a half-line, \cite{GL2}, see also \cite{BDM}. Moreover,  the support of $\nu$ can be characterized in terms of the support of $\mu$. No more is needed.
For $(a,b)\in \R^+\times \R$ such that $a\not=1$,  we denote by $x(a,b)$ the fixed point of the action of $(a,b)$. That is $x(a,b)$ is the unique point such that
$$
a\cdot x(a,b) + b = x(a,b).
$$  Then
$$
x(a,b) = \frac {b}{1-a}
$$
The following result was proved in \cite{GL2} (see also Proposition 2.5.4 in \cite{BDM})
\begin{lem}
Assume that
\begin{equation}\label{eq:h1}
\mbox{ there are } (a_1,b_1),(a_2,b_2)\in {\rm supp} \mu \mbox{ such that }
a_1>1, \ a_2 < 1 \ \mbox{ and }\ x(a_1,b_1)< x(a_2,b_2).
\end{equation}
Then there is a constant $c$ such that the support of $\nu $ contains the half-line  $[c, \8)$.

On the other hand if $\P[A=1, B>0]=0$  and for all $(a_1,b_1), (a_2,b_2)\in {\rm supp} \mu $ such that
$a_1>1$,  $a_2 < 1$ we have  $$x(a_2,b_2)\le x(a_1,b_1)$$ then the support of $\nu $ is contained in $(-\8 , c]$ for some $c\in \R$.

\end{lem}
\begin{rem}
Notice that if $\P[A=1, B>0]>0$ then the support of $\nu $ is always unbounded at $\infty $.
\end{rem}

\subsection{Letac's recursion}
One of the recursions considered by Goldie \cite{G} was ``so called''  ``Letac model'', see also Letac \cite{L}:
$$
\wt X_n = B_n + A_n\max\big\{\wt X_{n-1}, C_n \big\}
= \max\big\{ A_n \wt X_{n-1} + B_n, A_n C_n + B_n\big\}, \quad n\ge 1.
$$
Clearly
\begin{equation}
\wt X_n\geq X_n.
\end{equation}
Under assumptions of Theorem \ref{thm: kesten goldie} plus $\E[A^\a |C|^\a]<\8$, Goldie proved that
$$\P[\wt X>t]\sim C_L t^{-\a} \qquad \mbox{ as } t\to\8,
$$ but he didn't obtained necessary and sufficient conditions for positivity of $C_L$. A sufficient condition for positivity of $C_L$ formulated there says that there is a constant $c$ such that $\P [B-c(1-A)\geq 0]=1$ and $\P [B-c(1-A)> 0] +\P [A(C-c)> 0]>0$. A simpler sufficient condition is given in \cite{CV}:
\begin{equation}\label{col}
\P [A>1, B> 0]>0\quad \mbox{or}\quad \P [A>1, B\geq 0, C>0]>0.
\end{equation}
However, the first part of \eqref{col} seems to be inaccurate in view of what we are going to prove below.

Due to Lemma \ref{mthm} it is sufficient to check when the support of $\wt \nu$ - the law of $\wt X$ is unbounded. We prove here an appropriate condition formulated in terms of  constants which can be explicitly computed knowing the law $\wt\mu$ of the triple $(A,B,C)$. Let
\begin{equation}
\label{eq:N}
\begin{split}
N_1 &= \sup \big\{ac+b:\; (a,b,c)\in {\rm supp } \wt\mu\big\},\\
N_2 &= \sup \big\{ x(a,b):\; (a,b,c)\in {\rm supp } \wt\mu \mbox{ and  } a<1\big\},\\
N_3 &= \inf \big\{ x(a,b):\; (a,b,c)\in {\rm supp } \wt\mu \mbox{ and  } a> 1\big\}.\\
\end{split}
\end{equation}
It may happen that $N_1 = \8$, $N_2 = \8$ or $N_3 = -\8$.

The following result holds (compare with Theorem 6.2 in \cite{G}).

\begin{thm}
\label{cor: cl} Assume $\P[A=1, B>0]=0$.
Then $C_L> 0 $ if and only if $N_3 < \max\{N_1,N_2\}=N$.
\end{thm}
\begin{rem}
If $\P[A=1, B>0]>0$ the due to $\wt X_n \geq X_n$, the support of $\wt \nu $ is unbounded at $\8 $ and $C_L>0$.
\end{rem}
\begin{rem}
It is not enough to assume $\P [A>1, B> 0]>0$ to have positivity of $C_L$ as it is claimed in \cite{CV}. Indeed, assume that $\wt\mu $ is supported on two points $(a_1,b_1,c_1)=(3,1,-1)$
and $(a_2,b_2,c_2)=(\frac{1}{2},-1,0)$ with probabilities $p,1-p$, $p>0$, such that
$p\log 3 +(1-p)\log \frac{1}{2}<0$. Then $\P [A>1, B> 0]>0$ but $N_3=-\frac{1}{2}$, $N_2=-2$, $N_1=-1$ and so $C_L$ cannot be positive.
\end{rem}
\begin{proof}[Proof of Theorem \ref{cor: cl}]
In view of Lemma \ref{mthm} it is sufficient to check whether the support of $\wt \nu$ is unbounded at $+\8$.

Assume  $N_3 \ge N$. We prove that  the half-line $(-\8,N]$ is $\wt\mu$-invariant.
First we observe that for every $(a,b,c)\in \supp \wt\mu $
\begin{equation}
\label{eq:13}
aN+b\le N.
\end{equation}
 Indeed, if $a\neq 1$ then
\begin{align*}
aN+b =& a\cdot x(a,b) + b + a(N - x(a,b))\\
 = & x(a,b) + a(N - x(a,b)) \\
= & x(a,b)(1-a) + aN\le N,
\end{align*}
 because if $a<1$, then $x(a,b)\le N_2 \le N$ and if $a> 1$, $x(a,b)\ge N_3\ge N$. If $a=1$ and $b\leq 0$ then \eqref{eq:13} holds.

Let $\wt \Psi _{(a,b,c)}(x)=\max \{ ax+b, ac+b\}$. Take any $(a,b,c)\in {\rm supp }\wt\mu$, then for any $x\le c$ 
$$
\wt \Psi _{(a,b,c)}(x) = ac+b \le N_1 \le N.
$$
For $c<x\le N$
$$
\wt \Psi _{(a,b,c)}(x) =ax+b \le aN+b \le N.
$$
Finally, we notice that $\supp \wt \nu $ is included in any $\supp \wt \mu $-invariant closed set $W$. Indeed, if $x\in W$ then $\wt \Psi (x)\in W$ a.s. so every $n$
$$
\wt \Psi _1\circ \dots \wt \Psi _n(x)\in W\quad a.s.$$
But $\wt X =_d\lim _{n\to \8}\wt \Psi _1\circ \dots\circ \wt \Psi _n(x)$ and so $\wt X \in W$ a.s.
Thus ${\rm supp \wt \nu}\subset (-\8,N]$ i.e. $\wt X\le N$ a.s.
This proves $C_L=0$.

\medskip

Assume now that $N_3<N$ and let $X_n^x $ be as in \eqref{eq:1} with the initial condition $x$. Then
\begin{equation}
\wt X_n^x\geq X_n^x\quad \mbox{a.s.}
\end{equation}
and so   $\wt X$ is stochastically larger than $X$,
where $X$ satisfies \eqref{eq:2}.
The same holds if we fix $(a,b,c)\in \supp \mu $ and we repeat both iterations i.e.
$$
\wt \Psi ^n _{(a,b,c)}(x)\geq \Psi ^n _{(a,b)}(x),$$ where $\Psi _{(a,b)}(x)=ax+b$.
We are going to consider two cases: $N_3<N_2$ and $N_3<N_1$. In the first case we can find $(a_1,b_1,c_1)$ and $(a_2,b_2,c_2)$ in the support of $\wt \mu $ such that $a_1>1$, $a_2<1$ and $x(a_1,b_1)<x(a_2,b_2)$. Then the support of $X$ contains a half-line $[a,\8 )$ and so, $\wt X$ is unbounded at $\8 $.

In the second case, let $x\in \supp \nu $ and let $(a,b,c)\in \supp \wt \mu $ be such that $ac+b>N_3$. Then $\wt \Psi _{(a,b,c)}(x)\in \supp \wt \nu $ and
$\wt \Psi _{(a,b,c)}(x)>N_3$. Take $(a_1,b_1,c_1)\in \supp \wt \mu $ such that $a_1>1$ and
$$
x(a_1,b_1)<\wt \Psi _{(a,b,c)}(x).$$
Then
\begin{align*}
\wt \Psi^n _{(a_1,b_1,c_1)}(\wt \Psi _{(a,b,c)}(x)) &\ge
\Psi^n _{(a_1,b_1)}(\wt \Psi _{(a,b,c)}(x))\\
&= a^n_1(\wt \Psi _{(a,b,c)}(x)-
  x(a_1,b_1)) + x(a_1,b_1)\to+\8.
\end{align*}
Since for every $n$, $\wt \Psi^n _{(a_1,b_1,c_1)}(\wt \Psi _{(a,b,c)}(x))\in {\rm supp } \wt\nu$, this set must be unbounded.
\end{proof}
A very particular form of the Letac recusion is considered in the literature, when $AC+B = 0$, that is
$$
\wt X_n = \max\big\{ A_n \wt X_{n-1} +B_n,0\big\}.
$$
Under assumptions of Theorem \ref{thm: kesten goldie}
$$\P[\wt X>t]\sim C_M t^{-\a} \qquad \mbox{ as } t\to\8.$$
Since this process has numerous applications (see e.g. \cite{BCDZ,CV}) positivity of $C_M$ is crucial. It is known that if $\P[A>1, B>0]>0$ then $C_M>0$, \cite{BCDZ,CV}. Here we  provide an optimal condition.
\begin{thm}
 Assume $\P[A=1, B>0]=0$.
Then, $C_M> 0 $ if an only if $N_3 < N_2$ or $\P[A>1, B>0]>0$.
\end{thm}
\begin{proof}
Notice that here $N_1=0$. According to Theorem \ref{cor: cl} $C_M > 0$ if and only if $N_3 < N_2$ or
$N_3 < 0$. But $N_3<0$ means that there is $(a,b,c)\in \supp \wt \mu $ such that $a>1$ and $x(a,b)<0$ that is exactly $\P[A>1, B>0]>0$.
\end{proof}
\subsection{Iterated function systems}
\label{section: Lipschitz}
 Alsmeyer \cite{Alsm2014} and Mirek \cite{Mi} studied tails of general IFS, as defined in \eqref{eq:4}. 
Mirek assumed additionally that $\Theta $ is a metric space and for every $x\in \R$, the function $\theta \mapsto \Psi (\theta , x)$ is continuous. Then \eqref{preserve} holds, see \cite{Mi} and so the minorisation \eqref{eq: A1} only on the support of $\nu $ is the right one. Moreover, $\Psi $ in \cite{Mi}
 is comparable to the affine recursion in the following sense:
 for a.e. $\Psi$ there is a random variable $(A,B)\in \R^+\times \R$ such that
\begin{equation}
\label{eq: aff}
Ax-B \le \Psi(x) \le Ax+B,\qquad \mbox{for } x\in {\rm supp}\nu,
\end{equation} where $\nu$ is the support of $X$.
This condition has a very natural geometrical interpretation.
It means that the graph of $\Psi$ lies between two lines $Ax-B$ and $Ax+B$.
This allows us to think that the recursion is  close to the affine recursion.

To get the idea what is the meaning of \eqref{eq: aff} the reader may think of the recursion $\psi (\te ,x)=\max\{Ax,
B\}$, where $\te=(A, B)\in\R^{+}\times\R=\Te$. Notice that if $X_0=x\geq 0$ then all the iterations stay positive which implies that  $\supp \nu \subset [0,\8 )$. We have then
$$
0\leq \max (Ax,B)-Ax\leq B^+, \quad x\geq 0.$$
Notice that for the max recursion
\eqref{eq: aff} is not satisfied on $\R $, but only on $[0, \8)\supseteq\supp\nu$.

In this setting Mirek proved an analogue of Theorem \ref{thm: kesten goldie}
\begin{thm}\label{HTthm}
Suppose that $\Theta $ is a metric space and for every $x\in \R$, the function $\theta \mapsto \Psi (\theta , x)$ is continuous.
Assume that $\psi $ satisfies \eqref{eq: aff} and  the random pair $(A,B)$ satisfies hypotheses of Theorem \ref{thm: kesten goldie}. Let $\E \log L(\Psi )<0$ and $\E \log ^+ |\Psi (\omega ,0)|<\8$. Then $\X$ has a heavy tail and
\begin{align*}
\lim_{t\rightarrow\8}t^{\al}\P[X>t]& = C_+,\\
\lim_{t\rightarrow\8}t^{\al}\P[X<-t]&= C_-.
\end{align*}
\end{thm}
In such generality positivity of $C_\8 = C_+ + C_-$ was proved only under very particular and not intuitive assumptions. Namely, $s_{\infty}=\sup \{ s: \E |A|^s<\infty \}$. If
 additionally the support of $\nu$ is unbounded, and one
of the following condition is satisfied:
\begin{align}
\label{HTthm7}&s_{\8}<\8 \ \ \ \ \mbox{and}\ \ \ \ \ \lim_{s\rightarrow s_{\8}}\frac{\E(|B|^s)}{\E |A|^s}=0,\\
\label{HTthm8}&s_{\8}=\8 \ \ \ \ \mbox{and}\ \ \ \ \ \lim_{s\rightarrow
\8}\left(\frac{\E(|B|^s)}{\E |A|^s)}\right)^{\frac{1}{s}}<\8,
\end{align}
then Mirek \cite{Mi} proved that $C_{\8}> 0$.

\medskip

Our Lemma \ref{mthm} implies
\begin{thm}
Under hypotheses of Theorem \ref{HTthm}, if the support of $\nu$ is unbounded at $\8$, then $C_+ > 0$.
\end{thm}

Finally let us mention, that the above results may hold beyond the assumption \eqref{eq: aff}. Applying an appropriate transform one can consider limiting behavior of stationary measures of many other IFS defined e.g. on finite intervals, including the random logistic transform, the stochastic Ricker model, random automorphisms of [0,1]; see \cite{Alsm2014,BB}.

\section{Non-Cramer settings}
All the examples presented in previous sections works under the Cramer condition, when hypotheses of Theorem \ref{thm: kesten goldie} are satisfied.
Nevertheless the method is valid in more general settings. What we really need in Section \ref{sec:main result}, is to compare the tails of $\ov \X$ and $M = \sup_n \Pi_n$.  Exactly the same proof gives
 \begin{lem}
 \label{mthm2}
 Assume that
 \begin{itemize}
\item $\E \log L(\Psi )<0$ and $\E \log ^+ |\Psi (\omega ,0)|<\8$
 \item \eqref{eq: A0} is satisfied and $\E \log A<0$, $\E \log ^+|B|<\8$;
\item  the law of $M$ behaves regularly at infinity in the following sense: for every $\d >0$ there is $C>0$ such that
\begin{equation}\label{eq: A3}
  \P[M> C t] \le \d \P[M>t]
\end{equation}
for large $t$;
\item the tail of $\ov \X$ is controlled at $-\8$ by the tail of $M$ that is there is $C>0$ such that
\begin{equation}\label{eq: A2}
\P[\ov \X<-t]\le C \P[M>t], \ t>0;
\end{equation}
   \item $\X$ is unbounded at $+\8$
 \end{itemize}
 Then there is $\eps>0$ such that for large $t$
 $$
 \P[\X>t] \ge \eps \P[M > t].
 $$
 \end{lem}

This lemma can be applied e.g. in the settings of a recent paper of Kevei \cite{K}. The random difference equation \eqref{eq:1} is considered there in two cases:
\begin{itemize}
\item[1)] if $\E \log A^\a = 1$ for some $\a>0$, but $\E A^\a \log^+ A = \8$;
\item[2)] there is $\a>0$ such that $\E A^\a < 1$, but $\E A^s = \8$ for all $s>\a$.
\end{itemize}
Then, under some more detailed assumptions, applying the renewal type argument, Kevei \cite{K} proved analogous results to Theorem \ref{thm: kesten goldie}. Of course with a slightly different asymptotic.

In the first case there is $c_+>0$ and a slowly varying function $l(x)$ such that
\begin{equation}
\P [M>x]\asymp c_+l(x)x^{-\a }
\end{equation}
and
\begin{align*}
\lim _{x\to \8} l(x)^{-1}x^{\a }
\P [\X >x]=&C_+\geq 0,\\
\lim _{x\to \8} l(x)^{-1}x^{\a }
\P [\X <-x]=&C_-\geq 0,
\end{align*}
with $C_++C_->0$. In the second case the results are analogous but some more is required to conclude that $l(x)$ is slowly varying. In any case, if $l(x)$ is slowly varying, then \eqref{eq: A3} and \eqref{eq: A2} are satisfied and under hypotheses \eqref{eq: fixed point} from Lemma \ref{mthm2} we may conclude strict positivity of $C_+$ or $C_-$.

\medskip

Finally notice that because of  condition \eqref{eq: A3}, our method cannot be applied e.g. to the case when the law of $\log M$ is subexponential, as considered by Dyszewski \cite{D}.

\end{document}